\documentclass[12pt]{amsart}
\usepackage{amssymb,latexsym}
\usepackage{amsfonts}
\usepackage{amsmath}
\usepackage[colorlinks,linkcolor=blue,anchorcolor=blue,citecolor=blue]{hyperref}
\usepackage{algorithm}
\usepackage{enumerate}
\usepackage{algpseudocode}
\usepackage{verbatim}
\usepackage{graphicx}

\newcommand\F{{\mathbb F}}

\newcommand\Z{{\mathbb Z}}

\newcommand\N{{\mathbb N}}

\newcommand\lcm{{\mathrm{lcm}}}
\newcommand\Ann{{\mathrm{Ann}}}

\newcommand\cD{{\mathcal D}}

\newcommand\bh{\mathbf{h}}
\newcommand\bj{\mathbf{j}}
\newcommand\bk{\mathbf{k}}
\newcommand\bx{\mathbf{x}}
\newcommand\ba{\mathbf{a}}

\newcommand\bc{\mathbf{c}}
\newcommand\bmu{\pmb{\mu}}
\newcommand\bP{\bar{P}}
\newcommand\Lm{\textrm{Lm}}

\newtheorem{theorem}{Theorem}[section]
\newtheorem{lemma}[theorem]{Lemma}

\theoremstyle{definition}
\newtheorem{definition}[theorem]{Definition}

\newtheorem{remark}[theorem]{Remark}

\numberwithin{equation}{section}

\begin{document}

\title[]{Polynomial functions in the residue class rings of Dedekind domains}

\author{Xiumei Li}
\address{School of Mathematical Sciences, Qufu Normal University, Qufu, Shandong, 273165, China}
\email{lxiumei2013@hotmail.com}

\author{Min Sha}
\address{Department of Computing, Macquarie University, Sydney, NSW 2109, Australia}
\email{shamin2010@gmail.com}





\begin{abstract}
In this paper, as an extension of the integer case, we define polynomial functions over the
residue class rings of Dedekind domains, and then we give canonical representations
and counting formulas for such polynomial functions.
In particular, we give an explicit formula for the number of polynomial functions over
the residue class rings of polynomials over finite fields.
\end{abstract}

\maketitle



\section{Introduction}	

\subsection{Motivation}

Let $m$ and $n$ be two positive integers.
In \cite{Chen1995} Chen has defined the concept of a polynomial function from $\Z/n\Z$ to $\Z/m\Z$
and has obtained an exact formula for the number of such polynomial functions,
which has been extended by Chen \cite{Chen1996} to
functions from $\Z/n_1\Z \times \cdots \times \Z/n_r\Z$ to $\Z/m\Z$.

\begin{definition}[Chen \cite{Chen1996}]
A function $f: \Z/n_1\Z \times \cdots \times \Z/n_r\Z \to \Z/m\Z$ is said to be a \textit{polynomial function}, if it is represented by a polynomial $F\in \Z[x_1,\ldots,x_r]$ such that
$$
f(a_1,\ldots,a_r) \equiv F(a_1,\ldots,a_r) \pmod{m}
$$
for $a_i=0,1,\ldots,n_i-1; i=1,2,\ldots,r$.
\end{definition}

However, it hasn't been proved in \cite{Chen1995,Chen1996} that
 the results therein do not depend on the choices of complete sets of residues modulo $n_i$
(in the above definition, the complete sets are $\{0,1,\ldots,n_i-1\}, i=1,2,\ldots,r$).
In this paper, we supplement this in a more general setting
by using the $P$-orderings defined and studied in \cite{Bhargava1997}.

In fact, for the case when $n=m$, there are many related earlier results which one could refer to \cite{Carlitz1964,KO1968,Kempner1921,MS1984,Sin1974}.
Especially, Bhargava \cite{Bhargava1997,Bhargava2000} has considerably enlarged the setting for polynomial functions by
replacing ``the residue class rings of $\Z$" with ``finite principal ideal rings".
For example, given a finite principal ideal ring $R_0$ and a subset $S \subseteq R_0$,  in \cite[Section 3]{Bhargava1997} there are canonical representations and counting formulas
for polynomial functions from $S$ to $R_0$; see also \cite[Theorem 18]{Bhargava1997} for the case of several variables.

In this paper, we want to generalize the above concept of polynomial function to the case of residue class rings of Dedekind domains,
as well as consider its canonical representation and counting formula by following the same strategy as in \cite{Bhargava1997,Chen1995,Chen1996}.
Different from \cite{Bhargava1997} but as in \cite{Chen1995,Chen1996}, we do not restrict polynomial functions within a single residue class ring.

\subsection{Our situation}

From now on, let $D$ be a Dedekind domain.
For any non-trivial ideal $I \subseteq D$ (that is, $I\ne 0, D$), let $D/I$ be the residue class ring of $D$ modulo $I$,
and let $\cD_I$ be a complete set of residues modulo $I$ such that $0 \in \cD_I$.
For any $a \in D$, for simplicity we still use $a$ to denote the residue class of $a$ modulo
$I$. The reader can distinguish them easily by context.

In the sequel, let $I_1,\ldots,I_r, J$ be non-trivial ideals of $D$ ($r\ge 1$).

\begin{definition}    \label{def:func1}
A function $f: D/I_1 \times \cdots \times D/I_r \to D/J$ is said to be a \textit{polynomial function}, if it is represented by a polynomial $F\in D[x_1,\ldots,x_r]$ such that
$$
f(b_1,\ldots,b_r) \equiv F(b_1,\ldots,b_r) \pmod{J}
$$
for any $(b_1,\ldots,b_r)\in \cD_{I_1} \times \cdots \times \cD_{I_r}$,
where $(b_1,\ldots,b_r)$ is considered as an element in $D^r$ when evaluating $F(b_1,\ldots,b_r)$.
\end{definition}

In this paper, the aim is to provide canonical representations for the polynomial functions from $D/I_1 \times \cdots \times D/I_r$ to $D/J$
and give counting formulas for such functions.
Furthermore, we want to show that the set of such polynomial functions and so the counting formulas
do not depend on the choices of $\cD_{I_1}, \ldots, \cD_{I_r}$.

As applications, we not only recover the main results in the integer case, but also obtain new results for the polynomial rings over finite fields.

\section{Preliminaries}

\subsection{More on our setting}

Let
$$
K=\lcm [I_1, \ldots, I_r, J]   \qquad   \textrm{and}   \qquad     R = D/K.
$$
For any non-trivial ideal $I$ of $D$ with $I\mid K$, we denote by $\bar{I}$ the residue of $I$ modulo $K$.
Note that we can view $D/I$ as a subset of $R$, because any two distinct residues modulo $I$
can naturally represent two distinct residues modulo $K$.
So, in the sequel we view $D/I_1, \ldots, D/I_r, D/J$ as subsets of $R$.
Then, in \cite[Theorem 18]{Bhargava1997} there is a canonical representation of
polynomial functions from $D/I_1 \times \cdots \times D/I_r$ to $R$.
However, our purpose here is to study polynomial functions from $D/I_1 \times \cdots \times D/I_r$ to $D/J$.
We use the strategy in \cite{Bhargava1997} to reach our objective.

Assume that we have a prime factorization:
\begin{equation*}  \label{eq:factor}
K = P_1^{e_1} \cdots P_n^{e_n},
\end{equation*}
where $P_1,\ldots,P_n$ are pairwise distinct prime ideals of $D$ and each $e_i$ is a positive integer.
Then, the prime ideals of $R$ are exactly $\bP_1,\ldots, \bP_n$, and we also have
\begin{equation}  \label{eq:Pi}
\bP_i^{e_i} = \bP_i^{e_i+1}, \quad i=1,\ldots, n.
\end{equation}

\subsection{$P$-orderings}

Let $\N$ be the set of non-negative integers.
Now we recall some basic concepts introduced in \cite{Bhargava1997}.
For any non-zero prime ideal $P$ of $R$ and any element $a\in R$, let $w_P(a)$ be the highest power of $P$ containing $a$
(where $w_P(0)$ is defined to be the zero ideal).
Then, given a non-empty subset $X$ of $R$, we obtain a so-called \textit{$P$-ordering of $X$} as follows: choose $b_0$ to be any element of $X$,
and for $k=1,2,\ldots$ choose $b_k\in X$ to minimize the exponent of the highest power of $P$ dividing
$$
(b_k-b_0)(b_k-b_1) \cdots (b_k - b_{k-1}).
$$
Given such a $P$-ordering $\{b_i\}_{i\in \N}$ of $X$, we define the \textit{associated $P$-sequence of $R$} by
$v_0(X,P)=R$ and
$$
v_k(X,P) = w_P((b_k-b_0)(b_k-b_1) \cdots (b_k - b_{k-1})), \quad k=1,2,\ldots.
$$
If $X$ is a finite set, then for any $k \ge |X|$ we must have that $v_k(X,P)$ is the zero ideal.
Moreover, by \cite[Lemma 3]{Bhargava1997}, for any integer $0 \le k < |X|$, $v_k(X,P) \ne R$ for only finitely many primes $P$.
By the construction of a $P$-ordering, we know that (see also \cite[Lemma 1]{Bhargava1997}):
\begin{equation}    \label{eq:wv}
w_P((a-b_0)(a-b_1)\cdots (a-b_{k-1})) \subseteq v_k(X,P), \quad \textrm{for any $a\in X$}.
\end{equation}

Furthermore, by \cite[Theorem 1]{Bhargava1997} we know that any two $P$-orderings of $X$ give the same associated $P$-sequence.
So, the sequence $\{v_k(X,P)\}_{k\in \N}$ does not depend on the choice of the $P$-ordering.
We also define the sequence of \textit{factorial ideals} corresponding to the pair $(X,R)$ by $v_0(X)=R$ and
$$
v_k(X) = \prod_{\textrm{non-zero prime ideal $P$}} v_k(X,P), \quad k=1,2,\ldots,
$$
which again do not depend on the choices of such $P$-orderings.
In particular, if $X$ is a finite set, then for any $k \ge |X|$, $v_k(X)$ is the zero ideal.

Note that the prime ideals of $R$ are exactly $\bP_1,\ldots,\bP_n$.
For each $1\le i \le r$ and $1\le l \le n$, let $\{a_{l,i,j}\}_{j\in \N}$ be a fixed $\bP_l$-ordering of $D/I_i$ such that $a_{l,i,0}=0$.
Using the Chinese Remainder Theorem, for each $1\le i \le r$ we construct a sequence $\{a_{i,j}\}_{j\in \N}$ of elements of $R$
such that $a_{i,0}=0$ and $a_{i,j} \equiv a_{l,i,j} \pmod{\bP_l^{e_l}}$ for all $1\le l \le n$ and $j\in \N$.
This is crucial for our deductions.
For the sequence of factorial ideals for $R$, we have
\begin{equation}    \label{eq:vk}
v_k(D/I_i) = \prod_{l=1}^{n} v_k(D/I_i,\bP_l), \quad k=0,1,2,\ldots.
\end{equation}
In fact, by construction, for any integer $k \ge 0$ we have
\begin{equation}    \label{eq:vki}
v_k(D/I_i) = \langle (a_{i,k}-a_{i,0})(a_{i,k}-a_{i,1}) \cdots (a_{i,k} - a_{i,k-1}) \rangle.
\end{equation}

\subsection{A basis}

Let $R[x_1,\ldots,x_r]$ be the polynomial ring of $r$ variables over $R$.
We now define a basis for $R[x_1,\ldots,x_r]$ over $R$.
We first define an ordering in $\N^r$.

\begin{definition}
For any $\bk=(k_1,\ldots,k_r)$ and $\bh=(h_1,\ldots,h_r)$,
we say that $\bk$ is less than $\bh$, denoted by $\bk < \bh$,
if there exists $j$ such that $k_j < h_j$ and $k_i=h_i$ for all $i<j$.
As usual, $\bk \le \bh$ means that $\bk<\bh$ or $\bk=\bh$.
\end{definition}

Note that the above ordering automatically gives an ordering for the monomials of the ring $R[x_1,\ldots,x_r]$,
which is used later on without indication.

\begin{definition}
For any $\bk=(k_1,\ldots,k_r)\in \N^r$, we define
$$
(\bx)_\bk = \prod_{i=1}^{r} (x_i)_{k_i},
$$
where $(x_i)_{k_i}$ is defined as follows:
$$
\textrm{$(x_i)_{k_i}=1$ if $k_i=0$, otherwise $(x_i)_{k_i} = \prod_{j=0}^{k_i-1} (x_{i}-a_{i,j})$.}
$$
\end{definition}

Clearly, the polynomials $(\bx)_\bk,\bk\in \N^r$, form an $R$-basis for $R[x_1,\ldots,x_r]$.

\subsection{Notation and convention}

Define
\begin{itemize}
\item $\mu(I_i,J)=$ the smallest positive integer $k$ such that $v_k(D/I_i)\subseteq \bar{J}$ if such $k$ exists; otherwise, put $\mu(I_i,J)=|D/I_i|$.
(Note that $\mu(I_i,J)$ can be equal to infinity.)

\item for any $R$-module $M$, $\Ann(M) = \{b\in R:\, bM=0\}$, which is the so-called \textit{annihilator} of $M$ as an $R$-module.

\item for any ideal $I$ of $R$, $\Ann_J(I) = \{b\in R:\, bI \subseteq \bar{J}\}$, which is the so-called \textit{ideal quotient} $(\bar{J}:I)$.
\end{itemize}

We indicate again that we view $D/I_1, \ldots, D/I_r, D/J$ as subsets of $R$.
We define a relation $\sim$ in $R[x_1,\ldots,x_r]$ as follows: for any $F,G\in R[x_1,\ldots,x_r]$,
$F \sim G$ if and only if $F$ and $G$ represent the same polynomial function from $D/I_1 \times \cdots \times D/I_r$ to $D/J$.

\section{Main results}

\subsection{A criterion}

We first determine under which condition every function from $D/I_1 \times \cdots \times D/I_r$ to $D/J$ is a polynomial function.

\begin{theorem}   \label{thm:criterion}
Every function $f: D/I_1 \times \cdots \times D/I_r \to D/J$ is a polynomial function if and only if for any $i=1,\ldots,r$ and any prime factor $P$ of $J$,
no two elements of $\cD_{I_i}$ are congruent modulo $P$.
\end{theorem}

\begin{proof}
We first prove the necessary part by contradiction.
Without loss of generality, suppose that there is a prime factor of $J$, say $P$, such that there exist $b_1,b_2\in \cD_{I_1}$ such that
$b_1 \equiv b_2 \pmod{P}$.
So, $(b_1,0,\ldots,0)$ and $(b_2,0,\ldots,0)$ are two distinct elements in $\cD_{I_1} \times \cdots \times \cD_{I_r}$.
Then, for any function $f: D/I_1 \times \cdots \times D/I_r \to D/J$, if it is represented by a polynomial $F\in R[x_1,\ldots,x_r]$,  we have
$$
F(b_1,0,\ldots,0)\equiv F(b_2,0,\ldots,0) \pmod{P}.
$$
This means that the choices of $f(b_1,0,\ldots,0)$ and $f(b_2,0,\ldots,0)$ are not independent.
This contradicts the assumption.

Now, we prove the sufficient part.
Under the assumption, for each $1\le i \le r$, we have $\gcd(c-d,J)=1$ for any $c,d\in \cD_{I_i}$ with $c \ne d$.
Then, $c-d$ gives a unit in $D/J$, and we use $b_{i,c,d}$ to denote the inverse of $c-d$ in $D/J$.

Let $f: D/I_1 \times \cdots \times D/I_r \to D/J$ be a function.
By Lagrange interpolation, for any $(c_1,\ldots,c_r) \in \cD_{I_1}\times \cdots \times \cD_{I_r}$ we have
$$
f(c_1,\ldots,c_r) = \sum_{(d_1,\ldots,d_r)\in \cD_{I_1}\times \cdots \times \cD_{I_r}} f(d_1,\ldots,d_r) \prod_{i=1}^{r}\prod_{s\in \cD_{I_i}, \, s\ne d_i} b_{i,d_i,s}(c_i-s).
$$
Thus, $f$ is a polynomial function represented by the polynomial
$$
F(x_1,\ldots,x_r) = \sum_{(d_1,\ldots,d_r)\in \cD_{I_1}\times \cdots \times \cD_{I_r}} f(d_1,\ldots,d_r) \prod_{i=1}^{r}\prod_{s\in \cD_{I_i}, \, s\ne d_i} b_{i,d_i,s}(x_i-s).
$$
This completes the proof.
\end{proof}

\subsection{Canonical representations}

We now state a canonical representation for a polynomial function from $D/I_1 \times \cdots \times D/I_r$ to $D/J$.

\begin{theorem}\label{thm:rep}
Let $f$ be a polynomial function from $D/I_1 \times \cdots \times D/I_r$ to $D/J$. Then, $f$ can be represented by a polynomial of the form
$$
 F=\sum_{\bk \ge 0}b_{\bk}(\bx)_{\bk},
 $$
 where the coefficients $b_\bk \in R$ with $\bk=(k_1,\ldots,k_r) $ are uniquely determined modulo $\Ann_J(v_{k_1}(D/I_1)\cdots v_{k_r}(D/I_r))$,
 and the sum is over the set of all $r$-tuples $\bk$ such that $k_i < \mu(I_i,J)$ for each $1\le i \le r$.
 \end{theorem}

\begin{remark}
 We emphasize that by Theorem \ref{thm:rep}, the set of polynomial functions from $D/I_1 \times \cdots \times D/I_r$ to $D/J$
 does not depend on the choices of such complete sets $\cD_{I_i}$ of residues and the choices of such orderings.
\end{remark}

In order to prove Theorem \ref{thm:rep}, we need to make some preparations.

\begin{lemma}\label{lemma1}
Let $F\in R[x_{1},\ldots,x_{r}]$ whose leading monomial is $\bx^{\bj}=x_1^{j_1} \cdots x_r^{j_r}$. Then
$$
F=\sum_{0\le \bk \le \bj}b_{\bk}(\bx)_{\bk},
$$
where the coefficients $b_{\bk}\in R$ are uniquely determined by $F$.
\end{lemma}
\begin{proof}
The result follows from the fact that $\{(\bx)_{\bk}:\,\bk\in \N^r\}$ is an $R$-basis of $ R[x_1,\ldots,x_r] $.
\end{proof}

\begin{lemma}\label{lemma2}
Let $\bk=(k_1,\ldots,k_r)\in \N^r$, and assume that $k_i\ge \mu(I_i,J)$ for some $i$ (automatically $\mu(I_i,J)$ is finite). Then, $(\bx)_{\bk}\sim 0 $.
\end{lemma}

\begin{proof}
By definition, it suffices to prove that $(x_i)_{k_i} \sim 0$, which is equivalent to $(c_i)_{k_i} \equiv 0  \pmod{J}$ for any $c_i \in D/I_i$.

By \eqref{eq:wv} and the construction of the sequence $\{a_{i,j}\}_{j\in \N}$, we know that $(c_i-a_{i,0})(c_i-a_{i,1})\cdots (c_i-a_{i,k_i-1}) \in v_{k_i}(D/I_i,\bar{P}_l)$
for any $c_i \in D/I_i$ and any $1\le l \le n$, and so by \eqref{eq:vk}, it is an element of $v_{k_i}(D/I_i)$ for any $c_i \in D/I_i$.
Since $k_i\ge \mu(I_i,J)$, we have $v_{k_i}(D/I_i) \subseteq \bar{J}$, and thus any such $(c_i-a_{i,0})(c_i-a_{i,1})\cdots (c_i-a_{i,k_i-1}) \in \bar{J}$, that is
$$
(c_i-a_{i,0})(c_i-a_{i,1})\cdots (c_i-a_{i,k_i-1}) \equiv 0  \pmod{J}.
$$
Hence, $(x_i)_{k_i} \sim 0$.
\end{proof}

\begin{lemma}\label{lemma3}
$\sum_{0\le \bk \le \bj}b_{\bk}(\bx)_{\bk}\sim 0$ if and only if $b_{\bk}(\bx)_{\bk}\sim 0$ for all $0\le \bk \le \bj$.
\end{lemma}
\begin{proof}
Clearly, we only need to show the necessity.

The necessity is trivial when $\bj=0$.
Now, we assume that $\bj>0$.
Suppose that
\begin{equation}   \label{eq:bkj}
\sum_{0\le \bk \le \bj}b_{\bk}(\bx)_{\bk}\sim 0.
\end{equation}
Then, we have
$$
\sum_{0\le \bk \le \bj}b_{\bk}(0)_{\bk} = b_0 + \sum_{0 < \bk \le \bj}b_{\bk}(0)_{\bk} \equiv 0   \pmod{J},
$$
which, together with $(0)_{\bk}=0$ for any $\bk > 0$ (because for each sequence $\{a_{i,j}\}_{j\in \N}$, the first term $a_{i,0}=0$),
yields that $b_{0}\equiv 0 \pmod{J}$. So, $b_{0}(\bx)_{0} \sim 0$.

Now we proceed by induction. Assume that there is $\bh=(h_1,\ldots,h_r)$ such that $b_{\bk}(\bx)_{\bk}\sim 0$ for all $\bk < \bh \le \bj$.
We shall show that $b_{\bh}(\bx)_{\bh}\sim 0$. From the induction hypothesis and the original condition \eqref{eq:bkj},
we have
\begin{equation}   \label{eq:bhj}
\sum_{\bh \le \bk \le \bj} b_{\bk}(\bx)_{\bk}\sim 0.
\end{equation}
By definition, choosing $\ba=(a_{1,h_1},\ldots,a_{r,h_r})$, we get $(\ba)_{\bk}=0$ for any $\bk > \bh$.
So, by \eqref{eq:bhj}  we obtain
 $$
 b_{\bh}(\ba)_{\bh}=b_{\bh}\prod_{i=1}^{r}(a_{i,h_i}-a_{i,0})\cdots (a_{i,h_i}-a_{i,h_i-1}) \equiv 0  \pmod{J}.
 $$
 Then, by \eqref{eq:wv} and the construction of the sequences $\{a_{i,j}\}_{j\in \N}, i=1,2,\ldots,r$, we have that
 for any $\bc=(c_1,\ldots,c_r) \in D/I_1 \times \cdots \times D/I_r$,
 $$
 b_{\bh}(\bc)_{\bh}=b_{\bh}\prod_{i=1}^{r}(c_i-a_{i,0})\cdots (c_i-a_{i,h_i-1}) \equiv 0  \pmod{J},
 $$
 that is, $b_{\bh}(\bx)_{\bh}\sim 0$.
This in fact completes the proof.
\end{proof}

\begin{lemma}\label{lemma4}
For any $\bk=(k_1,\ldots,k_r) \in \N^r$, $b_{\bk}(\bx)_{\bk}\sim 0 $ if and only if
$$
b_{\bk} \in \Ann_J(v_{k_1}(D/I_1)\cdots v_{k_r}(D/I_r)).
$$
\end{lemma}

\begin{proof}
First, suppose that $b_{\bk}(\bx)_{\bk}\sim 0$. Let $\ba=(a_{1,k_1},\ldots,a_{r,k_r})$.
Then, in particular, we have $b_{\bk}(\ba)_{\bk} \equiv 0  \pmod{J}$, that is
$$
b_{\bk}(\ba)_{\bk}=  b_{\bk}\prod_{i=1}^{r}(a_{i,k_i}-a_{i,0})\cdots (a_{i,k_i}-a_{i,k_i-1}) \in \bar{J}.
$$
Besides, by \eqref{eq:vki} we have
$$
v_{k_1}(D/I_1)\cdots v_{k_r}(D/I_r) = \langle \prod_{i=1}^{r}(a_{i,k_i}-a_{i,0})\cdots (a_{i,k_i}-a_{i,k_i-1}) \rangle.
$$
Hence, we obtain
$$
b_{\bk} \in \Ann_J(v_{k_1}(D/I_1)\cdots v_{k_r}(D/I_r)).
$$

Conversely, if
$$
b_{\bk} \in \Ann_J(v_{k_1}(D/I_1)\cdots v_{k_r}(D/I_r)),
$$
we have
$$
b_{\bk} v_{k_1}(D/I_1)\cdots v_{k_r}(D/I_r) \subseteq \bar{J},
$$
that is,
$$
b_{\bk}\prod_{i=1}^{r}(a_{i,k_i}-a_{i,0})\cdots (a_{i,k_i}-a_{i,k_i-1}) \in \bar{J}.
$$
Then, as before,  it follows from \eqref{eq:wv} and the construction of the sequences $\{a_{i,j}\}_{j\in \N}, i=1,2,\ldots,r$ that $b_{\bk}(\bx)_{\bk}\sim 0$.
\end{proof}

Now, we are ready to prove Theorem \ref{thm:rep}.

\begin{proof}[Proof of Theorem \ref{thm:rep}]
Let $G$ be an arbitrary polynomial representation of $f$.
By Lemmas \ref{lemma1} and \ref{lemma2}, we have
$$
G \sim \sum_{\bk \ge 0}b_{\bk}(\bx)_{\bk},
$$
where the sum is over the set of all $r$-tuples $\bk$ such that $k_i < \mu(I_i,J)$ for each $1\le i \le r$.
It follows directly from Lemmas \ref{lemma3} and \ref{lemma4} that the above coefficients $b_\bk \in R$ are uniquely determined modulo
$\Ann_J(v_{k_1}(D/I_1)\cdots v_{k_r}(D/I_r))$, as desired.

\end{proof}

Moreover, we can get a simpler canonical representation for a polynomial function from $D/I_1 \times \cdots \times D/I_r$ to $D/J$.
Each monomial $x_1^{k_1}\cdots x_r^{k_r}\in R[x_1,\ldots,x_r]$ corresponds
to an $r$-tuple $(k_1,\ldots,k_r)\in \N^r$.
For any $F \in R[x_1,\ldots,x_r]$, let $\Lm(F)$ be the $r$-tuple
corresponding to the leading monomial of $F$.

\begin{theorem} \label{thm:rep2}
Let $f$ be a polynomial function from $D/I_1 \times \cdots \times D/I_r$ to $D/J$, and denote $\bx^{\bk}=x_1^{k_1}\cdots x_r^{k_r}$. Then, $f$ can be represented by a polynomial of the form
$$
 F=\sum_{\bk \ge 0}b_{\bk}\bx^{\bk},
 $$
 where the coefficients $b_\bk \in R$ with $\bk=(k_1,\ldots,k_r) $ are uniquely determined modulo $\Ann_J(v_{k_1}(D/I_1)\cdots v_{k_r}(D/I_r))$,
 and the sum is over the set of all $r$-tuples $\bk$ such that $k_i < \mu(I_i,J)$ for each $1\le i \le r$.
\end{theorem}

\begin{proof}
From Theorem \ref{thm:rep}, we know that $f$ can be represented by a polynomial $G$ of the form
$$
G = \sum_{0 \le \bk \le \bj}b_{\bk}(\bx)_{\bk}, \quad b_\bj \ne 0,
$$
where the coefficients $b_\bk \in R$ with $\bk=(k_1,\ldots,k_r) $ are uniquely determined modulo $\Ann_J(v_{k_1}(D/I_1)\cdots v_{k_r}(D/I_r))$,
and for $\bj=(j_1,\ldots,j_r)$ we have  $j_i < \mu(I_i,J)$ for each $1\le i \le r$.
We prove the desired result by induction on $\bj=\Lm(G)$.
If $\bj=0$, we are done.
Otherwise, let $H=G-b_{\bj} \bx^{\bj}$. Then, $\Lm(H) < \bj$.
So, by the induction hypothesis, $H$ can be represented by a polynomial of the desired form,
and then so is $G$ (that is, $f$).
\end{proof}

\subsection{The number of polynomial functions}

From Theorem \ref{thm:rep}, we can obtain the following counting formula when $D$ has the \textit{finite norm property}
(that is, for any non-zero ideal $I$ of $D$, $D/I$ is a finite ring).
Notice that if $D$ has the finite norm property, then for each $1\le i \le r$ we have
$$
1 \le \mu(I_i,J) \le |D/I_i|.
$$

\begin{theorem}\label{thm:number}
Assume that $D$ has the finite norm property. Define the $r$-tuple
$$
\bmu = (\mu(I_1,J)-1, \ldots, \mu(I_r,J)-1) \in \N^r.
$$
Then, the number of polynomial functions from $D/I_1 \times \cdots \times D/I_r$ to $D/J$ is given by
$$
N(I_1,\ldots,I_r;J)=\prod_{0 \le \bk \le \bmu} | (v_{k_1}(D/I_1)\cdots v_{k_r}(D/I_r)+\bar{J})/\bar{J}|.
$$
\end{theorem}

\begin{proof}
Note that by Theorem \ref{thm:rep}, it suffices to prove that
$$
R/\Ann_J(v_{k_1}(D/I_1)\cdots v_{k_r}(D/I_r))  \cong (v_{k_1}(D/I_1)\cdots v_{k_r}(D/I_r)+\bar{J})/\bar{J}
$$
as $R$-modules.
First, by definition we have
$$
\Ann_J(v_{k_1}(D/I_1)\cdots v_{k_r}(D/I_r)) = \Ann((v_{k_1}(D/I_1)\cdots v_{k_r}(D/I_r)+\bar{J})/\bar{J}).
$$
Since $R$ is a non-trivial quotient of a Dedekind domain, $R$ is a principal ideal ring.
Then, as an $R$-module, $(v_{k_1}(D/I_1)\cdots v_{k_r}(D/I_r)+\bar{J})/\bar{J}$ is a cyclic $R$-module.
So, we directly have
$$
R/\Ann((v_{k_1}(D/I_1)\cdots v_{k_r}(D/I_r)+\bar{J})/\bar{J})  \cong (v_{k_1}(D/I_1)\cdots v_{k_r}(D/I_r)+\bar{J})/\bar{J}
$$
as $R$-modules. This in fact completes the proof.
\end{proof}

We point out that if $r=1$ and $K=J$, then we recover the counting formula in \cite[Theorem 5]{Bhargava1997}.

We emphasize again that by Theorem \ref{thm:number}, the number of such polynomial functions
does not depend on the choices of such complete sets $\cD_{I_i}$ of residues and the choices of such orderings.
Hence, in order to obtain more explicit formulas for some special cases, we can make suitable choices; see Theorem \ref{thm:number1} and Section \ref{sec:app}.

\begin{remark}
There are several kinds of Dedekind domains having the finite norm property (see also \cite{CL}):
(i) the ring of integers of an algebraic number field;
(ii) the ring of integers of an algebraic function field;
(iii) the ring of integers of a non-Archimedean local field.
\end{remark}

In Theorem \ref{thm:number}, if we further assume that $I_1=\cdots =I_r=J$, then we can obtain a more explicit formula.

\begin{theorem}\label{thm:number1}
Assume that $D$ has the finite norm property.
Let $J$ be a non-trivial ideal of $D$ with prime factorization $J=P_1^{e_1}\cdots P_n^{e_n}$.
Let $R=D/J$.
For each $1\le i \le n$, let $N_i=|D/P_i|$, and let $\bmu_i$ be the minimal $r$-tuple $\bk=(k_1,\ldots,k_r) \in \N^r$ such that
$v_{k_1}(R)\cdots v_{k_r}(R) \subseteq \bP_i^{e_i}$.
Then, the number of $r$-ary polynomial functions from $R \times \cdots \times R$ to $R$ is given by
$$
N(J)=\prod_{i=1}^{n} \prod_{0\le \bk < \bmu_i} N_i^{e_i-\sum_{l=1}^{r}\sum_{j\ge 1} \lfloor k_l/ N_i^j\rfloor}.
$$
\end{theorem}

\begin{proof}
Note that the prime ideals of $R$ are exactly $\bar{P}_1,\ldots, \bar{P}_n$.
As in proving \cite[Corollary 2]{Bhargava1997}, we now recall the construction in \cite[Example 3]{Bhargava1997}.
For each $\bar{P}_i$, $|R/\bar{P}_i| = |D/P_i|=N_i$, and let
$$
\{a_{i,0}=0,a_{i,1},\ldots, a_{i,N_i-1}\}
$$
be a complete set of residues modulo $\bar{P}_i$.
Pick an element $\pi_i \in \bar{P}_i \setminus \bar{P}_i^2$. For any $j\in \N$, write
$$
j= c_0 + c_1N_i + \cdots + c_h N_i^h, \quad 0 \le c_0,\ldots, c_h < N_i,
$$
and define
$$
a_{i,j} = a_{i,c_0} + a_{i,c_1} \pi_i + \cdots + a_{i,c_h} \pi_i^h.
$$
Then, $\{a_{i,j}\}_{j\in \N}$ is a $\bar{P}_i$-ordering of $R$, and the associated $\bar{P}_i$-sequence of $R$ for $k< |R|$ is given by
$$
v_k(R,\bP_i) = \bar{P}_i^{\sum_{j\ge 1} \lfloor k/ N_i^j\rfloor},
$$
and for any $k\ge |R|$, $v_k(R,\bP_i)$ is the zero ideal.
So, we have
$$
v_k(R) = \prod_{i=1}^{n} \bar{P}_i^{\sum_{j\ge 1} \lfloor k/ N_i^j\rfloor}.
$$

Thus, for any $r$-tuple $\bk=(k_1,\ldots,k_r)\in \N^r$, we obtain
$$
v_{k_1}(R)\cdots v_{k_r}(R) = \prod_{i=1}^{n} \bar{P}_i^{\sum_{l=1}^{r}\sum_{j\ge 1} \lfloor k_l/ N_i^j\rfloor}.
$$
Hence,
$$
|v_{k_1}(R)\cdots v_{k_r}(R)| = | \prod_{i=1}^{n} P_i^{\sum_{l=1}^{r}\sum_{j\ge 1} \lfloor k_l/ N_i^j\rfloor} / \prod_{i=1}^{n} P_i^{e_i} |.
$$
If $\sum_{l=1}^{r}\sum_{j\ge 1} \lfloor k_l/ N_i^j\rfloor \le e_i$ for each $1\le i \le n$, we have
\begin{equation}    \label{eq:vk1r}
|v_{k_1}(R)\cdots v_{k_r}(R)| = \prod_{i=1}^{n} N_i^{e_i-\sum_{l=1}^{r}\sum_{j\ge 1} \lfloor k_l/ N_i^j\rfloor}.
\end{equation}
When computing $|v_{k_1}(R)\cdots v_{k_r}(R)|$ for other cases, we only need to note that $\bP_i^{e_i+1} = \bP_i^{e_i}$ for each $1\le i \le n$.

Therefore, by Theorem \ref{thm:number} and \eqref{eq:vk1r}, the number of such polynomial functions is
$$
N(J) = \prod_{i=1}^{n} \prod_{0\le \bk < \bmu_i} N_i^{e_i-\sum_{l=1}^{r}\sum_{j\ge 1} \lfloor k_l/ N_i^j\rfloor}.
$$
\end{proof}

We remark that if $r=1$, then we recover the formula in \cite[Corollary 2]{Bhargava1997}.

\section{Applications}    \label{sec:app}

Here, we use our general results to study two special cases. One is $D=\Z$, and the other is $D=\F_q[t]$ (polynomial ring),
where $\F_q$ is the finite field of $q$ elements.

\subsection{Case of $\Z$}

As in the setting of \cite{Chen1996}, let $n_1, \ldots, n_r, m$ be positive integers.
Then, we consider polynomial functions from $\Z/n_1\Z \times \cdots \times \Z/n_r\Z$ to $\Z/m\Z$.
Using Theorems \ref{thm:rep} and \ref{thm:number}, we indeed can recover the main results in \cite{Chen1996}.

Let $a_k=k$ for each $k\in \N$. Then, the sequence $\{a_k\}_{k\in \N}$ is a $p$-ordering of $\Z$ for any prime $p$; see \cite[Example 2]{Bhargava1997}.
So, we can use this ordering simultaneously.
Particularly, now for each $\bk=(k_1,\ldots,k_r) \in \N^r$ we have
$$
(\bx)_\bk = \prod_{i=1}^{r} x_i(x_i-1)\cdots (x_i-k_i+1).
$$

Let $\lambda(m)$ be the smallest positive integer $k$ such that $m \mid k!$.
Since $(a_k-a_0)\cdots (a_k-a_{k-1})=k!$ and $|\Z/n_i\Z|=n_i$ for each $1\le i \le r$,  we have
$v_k(\Z/n_i\Z) = \langle k! \rangle$ for any $0\le k<n_i$ and $v_k(\Z/n_i\Z)=0$ for any $k\ge n_i$.
Thus, for each $1\le i \le r$ we obtain
$$
\mu(n_i\Z,m\Z) = \min(n_i, \lambda(m)),
$$
which is denoted by $\mu(n_i,m)$ for simplicity.

Then, for each $\bk=(k_1,\ldots, k_r) \in \N^r$ with $k_i < \mu(n_i,m)$ for each $1\le i \le r$, we have
$$
\Ann_{m\Z}(v_{k_1}(\Z/n_1\Z)\cdots v_{k_r}(\Z/n_r\Z)) = \langle \frac{m}{\gcd(m,\prod_{i=1}^{r}k_i!)} \rangle.
$$

Hence, using Theorems \ref{thm:rep} and \ref{thm:number}, we recover \cite[Theorems 1 and 2]{Chen1996}.

\begin{theorem}[Chen \cite{Chen1996}]
Let $f$ be a polynomial function from $\Z/n_1\Z \times \cdots \times \Z/n_r\Z$ to $\Z/m\Z$.
Then, $f$ can be uniquely represented by a polynomial
$$
F= \sum_{\bk} b_\bk (\bx)_\bk,
$$
where the coefficients $b_\bk$ are integers satisfying
$$
0 \le b_\bk < \frac{m}{\gcd(m,\prod_{i=1}^{r}k_i!)},
$$
and the summation is taken over all $\bk=(k_1,\ldots,k_r)\in \N^r$ with $0\le k_i < \mu(n_i,m)$ for each $1 \le i \le r$.
\end{theorem}

\begin{theorem}[Chen \cite{Chen1996}]
The number of polynomial functions from $\Z/n_1\Z \times \cdots \times \Z/n_r\Z$ to $\Z/m\Z$ is given by
$$
N(n_1,\ldots,n_r;m)  = \prod_{\bk} \frac{m}{\gcd(m,\prod_{i=1}^{r}k_i!)},
$$
where the summation is taken over all $\bk=(k_1,\ldots,k_r)\in \N^r$ with $0\le k_i < \mu(n_i,m)$ for each $1 \le i \le r$.
\end{theorem}

\subsection{Case of $\F_q[t]$}

Denote $A=\F_q[t]$.
Let $f_1, \ldots, f_r, g$ be non-constant polynomials in $A$.
We consider polynomial functions from $A/f_1A \times \cdots \times A/f_rA$ to $A/gA$.

We write $\F_q=\{a_{0}=0,a_{1},\ldots,a_{q-1}\}$. For every $k\in \N$, write
$$
k= c_0 + c_1q + \cdots + c_h q^h, \quad 0 \le c_0,\ldots, c_h < q,
$$
and define
$$
a_{k}=a_{c_{0}}+a_{c_{1}}t+\ldots+a_{c_{h}}t^{h}.
$$
As indicated in \cite[Section 10]{Bhargava2000} (see also the example in page 289 of \cite{Adam}),
 the sequence $\{a_k\}_{k\in \N}$ is a $P$-ordering of $A$ for any prime ideal $P$ of $A$.
So, we use this ordering simultaneously.
Particularly, now for each $\bk=(k_1,\ldots,k_r) \in \N^r$ we have
$$
(\bx)_\bk = \prod_{i=1}^{r} (x_i-a_0)(x_i-a_1)\cdots (x_i-a_{k_i-1}).
$$
When $r=1$, this can define factorials for $A$, as an analogue of factorials of the rational integers;
see \cite{LS} for another analogue.

Let $\lambda(g)$ be the smallest positive integer $k$ such that
$$
g \mid \prod_{j=0}^{k-1}(a_k-a_j).
$$
Since $|A/f_iA|=q^{\deg f_i}$, for each $1\le i \le r$, we have
$$
v_k(A/f_iA) = \langle \prod_{j=0}^{k-1}(a_k-a_j) \rangle
$$
 for any $0\le k<q^{\deg f_i}$, and $v_k(A/f_iA)=0$ for any $k\ge q^{\deg f_i}$.
Thus, for each $1\le i \le r$ we obtain
$$
\mu(f_iA,gA) = \min(q^{\deg f_i}, \lambda(g)),
$$
which is denoted by $\mu(f_i,g)$ for simplicity.

Then, for each $\bk=(k_1,\ldots, k_r) \in \N^r$ with $k_i < \mu(f_i,g)$ for any $1 \le i \le r$, we have
$$
\Ann_{gA}(v_{k_1}(A/f_1A)\cdots v_{k_r}(A/f_rA)) = \langle \frac{g}{\gcd(g,\prod_{i=1}^{r}\prod_{j=0}^{k_i-1}(a_{k_i}-a_j))} \rangle.
$$

Hence, using Theorems \ref{thm:rep} and \ref{thm:number}, we obtain the following two results.

\begin{theorem}
Let $f$ be a polynomial function from $A/f_1A \times \cdots \times A/f_rA$ to $A/gA$. Then, $f$ can be uniquely represented by a polynomial
$$
 F=\sum_{\bk }b_{\bk}(\bx)_{\bk},
 $$
 where the coefficients $b_\bk \in A$ satisfy
 $$
 \ b_\bk=0 \text{\ \ \ or\ \ } \deg b_{\bk}< \deg\frac{g}{\gcd(g,\prod_{i=1}^{r}\prod_{j=0}^{k_i-1}(a_{k_i}-a_j))},
 $$
 and the summation is taken over all $\bk=(k_1,\ldots,k_r)\in \N^r$ with $0\le k_i < \mu(f_i,g)$ for each $1 \le i \le r$.
\end{theorem}

\begin{theorem}   \label{thm:poly}
The number of polynomial functions from $A/f_1A \times \cdots \times A/f_rA$ to $A/gA$ is given by
$$
N(f_1,\ldots,f_r;g) = \prod_{\bk}q^{\deg \frac{g}{\gcd(g,\prod_{i=1}^{r}\prod_{j=0}^{k_i-1}(a_{k_i}-a_j))} },
$$
and the summation is taken over all $\bk=(k_1,\ldots,k_r)\in \N^r$ with $0\le k_i < \mu(f_i,g)$ for each $1 \le i \le r$.
\end{theorem}

We remark that Theorem~\ref{thm:poly} plays a key role in studying
 congruence preserving functions in the residue class rings of polynomials over finite fields in \cite{LS2},
 which is an analogue of the integer case \cite{Bhargava1997-1}.

\section*{Acknowledgments}

The research of the first author was supported by National Science Foundation of China Grant No. 11526119
and Scientific Research Foundation of Qufu Normal University No. BSQD20130139.
The research of the second author was supported by a Macquarie University Research Fellowship.

\end{document}